\documentclass[a4paper,12pt]{article}

\usepackage{amsmath, amssymb, amsthm}
\usepackage{hyperref}

\usepackage{enumerate}
\usepackage{bm}
\usepackage{graphicx}

\usepackage{pgf,tikz}
\usetikzlibrary{arrows}

\usepackage{color}

\usepackage{algorithm, algorithmic}

\newcommand{\mS}{\mathcal{S}}
\newcommand{\mP}{\mathcal{P}}
\newcommand{\mL}{\mathcal{L}}
\newcommand{\mI}{\mathcal{I}}
\newcommand{\mV}{\mathcal{V}}
\newcommand{\mN}{\mathcal{N}}
\newcommand{\mH}{\mathcal{H}}

\newcommand{\mO}{\mathcal{O}}
\newcommand{\N}{\mathbb{N}}
\newcommand{\dist}{\mathrm{d}}

\newcommand{\comp}{\overline}

\newtheorem{thm}{Theorem}[section]

\newtheorem{lem}[thm]{Lemma}
\newtheorem{cor}[thm]{Corollary}

\newtheorem*{main1}{Main Result 1}
\newtheorem*{main2}{Main Result 2}

\begin{document}

\title{On semi-finite hexagons of order $(2,t)$ containing a subhexagon}
\author{Anurag Bishnoi and Bart De Bruyn}
\maketitle

\begin{abstract}
The research in this paper was motivated by one of the most important open problems in the theory of generalized polygons, namely the existence problem for semi-finite thick generalized polygons. We show here that no semi-finite generalized hexagon of order $(2,t)$ can have a subhexagon $H$ of order $2$. Such a subhexagon is necessarily isomorphic to the split Cayley generalized hexagon $H(2)$ or its point-line dual $H^D(2)$.  In fact, the employed techniques allow us to prove a stronger result. We show that every near hexagon $\mathcal{S}$ of order $(2,t)$ which contains a generalized hexagon $H$ of order $2$ as an isometrically embedded subgeometry must be finite. Moreover, if $H \cong H^D(2)$ then $\mathcal{S}$ must also be a generalized hexagon, and consequently isomorphic to either $H^D(2)$ or the dual twisted triality hexagon $T(2,8)$.
\end{abstract}

\bigskip \noindent \textbf{Keywords:} generalized hexagon, near hexagon, valuation\\ 
\textbf{MSC2000:} 51E12, 05B25

\section{Introduction and overview} \label{sec1}

All considered point-line geometries in this paper are {\em partial linear spaces}, these are geometries having the property that every two distinct points are incident with at most one line. The distance between two points $x_1$ and $x_2$ of a partial linear space $\mathcal{S}$ will always be measured in the collinearity graph. This distance will be denoted by $\dist_{\mS}(x_1,x_1)$, or shortly by $\dist(x_1,x_2)$ if no confusion could arise. If $\mathcal{S}$ is a subgeometry of another partial linear space $\mathcal{S}'$, then $\mathcal{S}$ is called {\em isometrically embedded into} $\mathcal{S}'$ whenever $\dist_{\mathcal{S}}(x,y) = \dist_{\mathcal{S}'}(x,y)$ for all points $x$ and $y$ of $\mathcal{S}$.

A \textit{near $2d$-gon} with $d \in \N$ is a partial linear space $\mathcal{S}$ that satisfies the following properties: 
\begin{enumerate} [(NP1)]
\item The collinearity graph of $\mathcal{S}$ is connected and has diameter $d$. 
\item For every point $x$ and every line $L$ there exists a unique point $\pi_L(x)$ incident with $L$ that is nearest to $x$. 
\end{enumerate} 
A {\em near polygon} is a near $2d$-gon for some $d \in \N$. Near polygons were introduced by Shult and Yanushka in \cite{SY}.  A near $0$-gon ($d = 0$) is just a point while a near $2$-gon ($d = 1$) is a line.  The class of near $4$-gons ($d = 2$) coincides with the class of possibly degenerate generalized quadrangles. Generalized quadrangles (GQ's) belong to the family of generalized polygons, an important class of point-line geometries introduced by Tits in \cite{Ti}.

All generalized polygons under investigation in this paper will be generalized hexagons. A \textit{generalized hexagon} can be viewed as a near hexagon which satisfies the following additional properties: 
\begin{enumerate}[(GH1)] 
\item Every point is incident with at least two lines.
\item Any two points at distance 2 have a unique common neighbor.
\end{enumerate}

A partial linear space is said to have \textit{order} $(s,t)$ if there are exactly $s+1$ points on each line and exactly $t+1$ lines through each point. If $s = t$ then we simply say that the geometry has {\em order} $s$. A partial linear space is called \textit{thick} if it has at least three points per line and at least three lines through each point. 

It can be shown that every (possibly infinite) thick generalized polygon has an order $(s,t)$. The standard reference for generalized polygons (cf.~\cite{VM}) contains the proof of this result among other standard results on generalized polygons. One of the important open problems in the theory of generalized polygons is the existence problem for semi-finite thick generalized polygons. These are the polygons which have finite number of points per line but infinitely many lines through each point. It has been shown that semi-finite GQ's of order $(s,t)$ do not exist for $s$ equal to $2$, $3$ or $4$ (\cite{Br,Ca,Ch}), but the problem remains open for other generalized polygons.  This problem has been stated as Problem 5 on the list of open problems on generalized polygons (cf.~\cite[p.~475]{VM}) where it is written that ``\textit{Any other result in this spirit would be an important step towards the solution of the problem. Especially tempting seems to be the case of a hexagon with s = 2.}'' In this paper we solve a very special case of this problem.

\begin{main1} \label{mr1}
Every near hexagon with three points on each line that contains a generalized hexagon $\mH$ of order $2$ as an isometrically embedded subgeometry is finite. In particular, there are no semi-finite generalized hexagons of order $(2,t)$ that contain a generalized hexagon of order 2 as a subgeometry.
\end{main1}

\medskip \noindent Note that if $\mathcal{S}$ and $\mathcal{S}'$ are two generalized hexagons such that $\mathcal{S}$ is a full subgeometry of $\mathcal{S}'$, then $\mathcal{S}$ is always isometrically embedded in $\mathcal{S}'$.

Finite generalized hexagons of order $(2,t)$ exist only if $t \in \{ 1,2,8 \}$. Any generalized hexagon of order $(2,1)$ is isomorphic to the point-line dual of the double of the Fano plane. In~\cite{CT} it was shown that the split Cayley generalized hexagon $H(2)$ and its point-line dual $H^D(2)$ are the only two generalized hexagons of order $2$ up to isomorphism. In~\cite{CT}, it was also shown that there exists up to isomorphism a unique generalized hexagon of order $(2,8)$, namely the dual twisted triality hexagon $T(2,8)$. We will also prove the following characterization of the latter generalized hexagon.

\begin{main2} \label{mr2}
Let $\mN$ be a (possibly infinite) near hexagon with three points on each line that has an isometrically embedded subgeometry isomorphic to $H^D(2)$. Then $\mN$ is a generalized hexagon and hence isomorphic to $H^D(2)$ or $T(2,8)$.
\end{main2}

\bigskip \noindent \textbf{Remarks.} (1) With the used techniques we were not able to show that no semi-finite generalized hexagon of order $(2,t)$ can have a subhexagon of order $(2,1)$.

(2) There are several examples of near hexagons that are not generalized hexagons and do contain $H(2)$ as a proper isometrically embedded full subgeometry (for instance, the dual polar spaces $DW(5,2)$ and $DH(5,4)$).

\bigskip \noindent The main tool used in the proofs will be that of valuations.  Different kinds of valuations were introduced in the papers \cite{bdb:Gn,bdb:polygonal,bdb-pvdc:1} and have since been used to obtain several classification results for near polygons, see \cite{bdb:octagon,bdb-pvdc:2}. We recall the basic notions of the theory of valuations in Section~\ref{sec2}. In Section \ref{sec3}, we give a description of the valuation geometry of the generalized hexagon $H^D(2)$ and use that information to prove those parts of the main results that involve a subgeometry isomorphic to $H^D(2)$. In Section \ref{sec4}, we realize similar goals for the generalized hexagon $H(2)$. The computer algorithms used for our computations are described in the appendix. These algorithms were implemented in the computer programming language GAP~\cite{gap}. The full GAP code we used is available online, see \cite{Bi-bdb}. 

\section{Preliminaries}  \label{sec2}

A partial linear space $\mathcal{S}$ will be viewed here as a triple $(\mP,\mL,\mI)$, where $\mP$ denotes the non-empty point set, $\mL$ the line set and $\mI \subseteq \mP \times \mL$ the incidence relation. A partial linear space $\mS = (\mP, \mL, \mI)$ is a \textit{subgeometry} of another partial linear space $\mS' = (\mP ', \mL ', \mI ')$ if $\mP \subseteq \mP'$, $\mL \subseteq \mL '$ and $\mI = \mI ' \cap (\mP  \times \mL)$. A subgeometry is called \textit{full} if for every line $L$ in $\mL$ the set $\{x \in \mP : x ~\mI~ L\}$ is equal to $\{x \in \mP' : x ~\mI'~ L\}$. Recall that if $d_{\mS}(x,y) =  d_{\mS'}(x,y)$ for every two points $x,y$ in $\mP$, then we will say that 
$\mS$ is isometrically embedded into $\mS'$. If $x$ is a point and $i \in \N$, then $\Gamma_i(x)$ denotes the set of points of $\mathcal{S}$ at distance $i$ from $x$. If $x$ is a point and $X$ a nonempty set of points, then $\dist(x,X)$ denotes the minimum distance from $x$ to a point of $X$. If $\emptyset \not= X \subseteq \mathcal{P}$ and $i \in \N$, then $\Gamma_i(X)$ denotes the set of points at distance $i$ from $X$. 

Suppose $\mathcal{N}$ is a near polygon. A {\em quad} of $\mathcal{N}$ is a set $Q$ of points satisfying:
\begin{enumerate}[(Q1)]
\item The maximum distance between two points of $Q$ is equal to 2.
\item If $x$ and $y$ are two collinear points of $Q$, then every point of the unique line through $x$ and $y$ is also contained in $Q$.
\item If $x$ and $y$ are two noncollinear points of $Q$, then every common neighbor of $x$ and $y$ also belongs to $Q$. 
\item The subgeometry of $\mathcal{N}$ determined by those points and lines that are contained in $Q$ is a nondegenerate generalized quadrangle. 
\end{enumerate}
Sufficient conditions for the existence of quads were given by Shult and Yanushka \cite{SY}. In Proposition 2.5 of that paper it was shown that if $a$ and $b$ are two points of a near polygon at distance $2$ from each other, and if $c$ and $d$ are two common neighbours of $a$ and $b$ such that at least one of the lines $ac$, $ad$, $bc$, $bd$ contains at least three points, then $a$ and $b$ are contained in a unique quad. 

Now, suppose that $\mathcal{N}$ is a near polygon with three points per line, and that $a$ and $b$ are two points at distance 2 from each other. If $a$ and $b$ have at least two common neighbors, then $a$ and $b$ are contained in a unique quad, which is a generalized quadrangle of order $(2,t)$. The number $t$ is finite (there are no semi-finite GQ's of order $(2,t)$) and equal to either $1$, $2$ and $4$, corresponding to the cases where the quad is a $(3 \times 3)$-grid, a GQ isomorphic to $W(2)$ or a GQ isomorphic to $Q(5,2)$, respectively. This implies the following.

\begin{lem} \label{lem2.1}
If $\mathcal{N}$ is a (possibly infinite) near polygon with three points on each line, then every two points at distance $2$ have either $1$, $2$, $3$ or $5$ common neighbors.  
\end{lem}

The main theme of this paper is to study and classify near $2d'$-gons $\mN'$ which contain a given sub-near-$2d$-gon $\mN$ isometrically embedded in them as a full subgeometry.  We do so with the help of certain integer valued functions defined on the points of $\mN$.

Let $\mN = (\mP, \mL, \mI)$ be a near $2d$-gon. A function $f: \mP \rightarrow \mathbb{Z}$ is called a \textit{semi-valuation} of $\mN$ if every line $L$ has a unique point $x_L$ with minimal $f$-value and every other point on $L$ has $f$-value $f(x_L) + 1$. In this paper, a {\em valuation}\footnote{In other papers, extra conditions are imposed for a semi-valuation to be a valuation.}    of $\mN$ will be a semi-valuation for which the minimal value is equal to 0.   If $f$ is a valuation of $\mN$, then we denote by $M_f$ the maximal value attained by $f$ and by $\mathcal{O}_f$ the set of points with value 0. The set of points of $\mN$ with non-maximal $f$-value is a {\em hyperplane} $H_f$ of $\mN$, that is, a set of points distinct from the whole point set having the property that each line has either one or all its points in it. Two (semi-)valuations $f_1$ and $f_2$ of $\mN$ will be called {\em isomorphic} if there exists an automorphism $\theta$ of $\mN$ such that $f_2 = f_1 \circ \theta$. We list some examples of valuations.
\begin{itemize}
\item Fix a point $p$ of $\mN$ and define $f$ by $f(x) = d(x,p)$.  It follows directly from (NP2) that $f$ is a valuation of $\mN$.  It is called the \textit{classical valuation with center $x$}. 
\item A subset $O$ of points is called an \textit{ovoid} of $\mN$ if every line of $\mN$ intersects $O$ in a single point.  Fix an ovoid $O$ of $\mN$ (if there exists one) and define $f$ by $f(x) = 0$ for all $x$ in $O$ and $f(x) = 1$ for every other point of $\mN$.  Then $f$ is a valuation which is called an \textit{ovoidal valuation}. 
\end{itemize}
Two valuations $f_1$ and $f_2$ of $\mN$ are called \textit{neighboring valuations} if there exists an $\epsilon \in \mathbb{Z}$ (necessarily belonging to $\{ -1,0,1 \}$) such that $|f_1(x) - f_2(x) + \epsilon| \leq 1$ for every point $x$ of $\mathcal{N}$. The number $\epsilon$ is uniquely determined, except when $f_1=f_2$, in which case there are three possible values for $\epsilon$, namely $-1$, $0$ and $1$. 

Suppose that every line of $\mN$ is incident with precisely three points. Suppose also that $f_1$ and $f_2$ are two neighboring valuations of $\mN$ and let $\epsilon \in \{ -1,0,1 \}$ such that $|f_1(x) - f_2(x) + \epsilon| \leq 1$ for every point $x$ of $\mathcal{N}$. If $x$ is a point such that $f_1(x) = f_2(x)-\epsilon$, then we define $f_3'(x) := f_1(x)-1 = f_2(x)-\epsilon-1$. If $x$ is a point of $\mathcal{N}$ such that $f_1(x) \not= f_2(x)-\epsilon$, then $f_3'(x)$ denotes the larger among $f_1(x)$ and $f_2(x)-\epsilon$. If we put $f_3(x) := f_3(x)-m$, where $m \in \{ -1,0,1 \}$ is the minimal value attained by $f_3'$, then it can be shown that $f_3$ is again a valuation, which we will  denote by $f_1 \ast f_2$. The map $f_1 \ast f_2$ is well-defined: if $f_1=f_2$, then there are three possibilities for $\epsilon$, but for each of them, we would have $f_1 \ast f_2 = f_1 = f_2$.  The following properties hold: (i) $f_2 \ast f_1 = f_1 \ast f_2 = f_3$; (ii) $f_1$ and $f_3$ are neighboring valuations and $f_1 \ast f_3 = f_2$; (iii) $f_2$ and $f_3$ are neighboring valuations and $f_2 \ast f_3 = f_1$. For more information about neighboring valuations and proofs of the above facts, we refer to \cite{bdb:Gn}. The following holds.

\begin{lem} \label{lem2.2}
If $\mN = (\mP, \mL, \mI)$ is a near polygon which is an isometrically embedded full subgeometry of a near polygon $\mN' = (\mP', \mL', \mI')$, then the following holds: 
\begin{enumerate}
\item[$(1)$] For every point $x$ in $\mP'$ the function $f_x: \mP \rightarrow \mathbb{N}$ defined by $f_x(y) := d(x,y) - d(x, \mP)$ is a valuation of $\mN$.
\item[$(2)$] For every pair of distinct collinear points $x_1$ and $x_2$ in $\mN'$, the valuations $f_{x_1}$ and $f_{x_2}$ are neighboring.
\item[$(3)$] If every line of $\mN'$ is incident with three points and if $\{ x_1,x_2,x_3 \}$ is a line of $\mN'$, then $f_{x_1} \ast f_{x_2} = f_{x_3}$. In particular, if two of $f_{x_1},f_{x_2},f_{x_3}$ coincide then they are all equal.
\end{enumerate}
\end{lem}
\begin{proof}
These properties were already (implicitly) proved in the literature. Claim (1) is a consequence of (NP2). Claim (2) follows from the fact that $|\dist(x,x_1)-\dist(x,x_2)| \leq 1$ for every point $x$ of $\mathcal{N}$. Claim (3) is again a consequence of (NP2). Indeed, if $x$ is a point of $\mathcal{N}$, then $\dist(x,x_3) = \dist(x,x_1)-1$ if $\dist(x,x_1)=\dist(x,x_2)$ and $\dist(x,x_3) = \max \{  \dist(x,x_1) , \dist(x,x_2) \}$ if $\dist(x,x_1) \not= \dist(x,x_2)$. 
\end{proof}

\section{Near hexagons containing an isometrically embedded $H^D(2)$}  \label{sec3}

In this section, we study near hexagons of order $(2,t)$ that contain the generalized hexagon $H^D(2)$ as an isometrically embedded subgeometry. 

\subsection{The valuation geometry of $H^D(2)$} \label{sec3.1}

The valuation geometry $\mathcal{V}$ of $H^D(2)$ is defined as the partial linear space whose points are the valuations of $H^D(2)$ and whose lines are the triples $\{ f_1,f_2,f_3 \}$, where $f_1$, $f_2$ and $f_3$ are three mutually distinct valuations of $H^D(2)$ such that $f_1$ and $f_2$ are neighboring valuations and $f_3 = f_1 \ast f_2$. We have used the computer algebra system GAP to determine all valuations and all lines of the valuation geometry $\mathcal{V}$, see \cite{Bi-bdb} and the appendix. Our results are summarized in Tables \ref{tab:1} and \ref{tab:2}. 

In Table~\ref{tab:1} we give a unique label, called ``Type'', to each isomorphism class of valuations. The ``$\#$'' column records the number of distinct valuations (points of $\mV$) of each type, i.e., the size of each class and the ``Value Distribution'' column records the number of points of $H^D(2)$ of value $i$ for every $i \in \{0, 1, 2, 3 \}$. In Table~\ref{tab:2}  we record the number of distinct lines of each type through a given point in $\mV$, where the type of a line of $\mV$ is the lexicographically ordered string of the types of the points incident with it.

\begin{table}[!htbp]
\begin{center}
\begin{tabular}{ | l || l | l | l | l | l | }
 \hline
  Type & $\#$ & $M_f$ & $|\mO_f|$ & $|H_f|$  & Value Distribution \\ \hline \hline
  $A$ & $63$ & $3$ & $1$ & $31$ & $[1, 6, 24, 32]$  \\ \hline
  $B$ & $252$ & $3$ & $1$ & $47$ & $[1, 14, 32, 16]$ \\ \hline
  $C$ & $252$ & $2$ & $1$ & $23$ & $[ 1, 22, 40, 0 ]$ \\ \hline
  $D$ & $1008$ & $2$ & $5$ & $31$ & $[5, 26, 32, 0]$ \\ \hline
\end{tabular}
\end{center}
\caption{The valuations of $H^D(2)$}
\label{tab:1}
\end{table}

\begin{table}[!htbp]
\begin{center}
\begin{tabular}{| l || l | l | l | l |}
\hline 
Type & $A$ & $B$ & $C$ & $D$ \\ \hline \hline

$AAA$ & $3$ & $-$ & - & - \\ \hline

$ABB$ & $2$ & $1$ & - & - \\ \hline 

$ACC$ & $2$ & - & 1 & - \\ \hline 

$ADD$ & $24$ & - & - & $3$ \\ \hline 

$BBB$ & - & $4$ & - & - \\ \hline

$BCC$ & - & $1$ & $2$ & - \\ \hline 

$BDD$ & - & $4$ & - & $2$ \\ \hline 

$CCC$ & - & - & 8 & - \\ \hline 

$CCD$ & - & - & $40$ & $5$ \\ \hline

$CDD$ & - & - & $4$ & $2$ \\ \hline

$DDD$ & - & - & - & $10$ \\ \hline   
\end{tabular}
\end{center}
\caption{The lines of the valuation geometry of $H^D(2)$} 
\label{tab:2}
\end{table}

\bigskip \noindent With the aid of a computer, we have also proved the following facts about the valuation geometry $\mathcal{V}$:

\begin{lem} \label{lem3.1}
Let $\mathcal{V}'$ be the subgeometry of $\mathcal{V}$ obtained by taking only the points of Type C and the lines of Type CCC. Then:
\begin{itemize}
\item[$(a)$] The geometry $\mathcal{V}'$ is connected.
\item[$(b)$] If $f_1$ and $f_2$ are distinct collinear points of $\mathcal{V}'$, then the unique points in $\mathcal{O}_{f_1}$ and $\mathcal{O}_{f_2}$ lie at distance 3 from each other.
\item[$(c)$] Suppose $G$ is a $(3 \times 3)$-subgrid of $\mathcal{V}'$. Let $f_1$ and $f_2$ be two noncollinear points of $G$, and let $x_i$ with $i \in \{ 1,2 \}$ denote the unique point with $f_i$-value $0$. Then $\dist(x_1,x_2)=3$.   
\end{itemize}
\end{lem}

\bigskip \noindent Observe that the valuations $f_1$ and $f_2$ mentioned in Lemma \ref{lem3.1} have Type C and so each of them has a unique point with value 0, see Table \ref{tab:1}. The geometry $\mathcal{V}'$ has $(3 \times 3)$-subgrids. In fact, our computer computations showed that through each point of $\mathcal{V}'$, there are precisely 16 $(3 \times 3)$-subgrids. The structure of the valuation geometry as described in Tables \ref{tab:1}, \ref{tab:2} and Lemma \ref{lem3.1} will suffice to derive the desired results without further assistance of a computer.    

\subsection{Proof of those main results that involve a subhexagon $H^D(2)$} \label{sec3.2}

In this section, $\mN$ denotes a near hexagon that contains a subhexagon $\mH$ isomorphic to $H^D(2)$ isometrically embedded in it as a full proper subgeometry. If $x$ is a point of $\mN$ then by Lemma \ref{lem2.2}(1) the map $y \mapsto f(x,y) - \dist(x,\mathcal{H})$ defines a valuation $\psi(x) := f_x$ of $\mH$, the so-called {\em valuation of} $\mH$ {\em induced by} $x$. The {\em type of a point} of $\mN$ is defined to be the type of the induced valuation. The {\em type of a line} $L$ of $\mN$ is defined to be the lexicographically ordered string of the types of the points incident with $L$. We denote by $\mathcal{V}$ the valuation geometry of $\mathcal{H}$. Points and lines of $\mathcal{V}$ will also be called {\em $\mathcal{V}$-points} and {\em $\mathcal{V}$-lines}. Recall that by Lemma \ref{lem2.2}(3), if $\{ x_1,x_2,x_3 \}$ is a line of $\mN$, then either $f_{x_1}=f_{x_2}=f_{x_3}$ or $\{ f_{x_1},f_{x_2},f_{x_3} \}$ is a line of $\mV$. If the latter case occurs, then we call $\psi(L) := \{ f_{x_1},f_{x_2},f_{x_3} \}$ the line of $\mV$ induced by $L$. We have: 

\begin{lem} \label{lem3.2}
Every point in $\mN$ has one of the four types mentioned in Table~\ref{tab:1}. For a line $L$ in $\mN$ if all the points on $L$ induce distinct valuations then $L$ has one of the eleven types mentioned in Table~\ref{tab:2}. 
\end{lem}

\begin{lem} \label{lem3.3}
\begin{enumerate}
\item[$(1)$] Every point of $\mN$ has distance at most $1$ from $\mH$. 
\item[$(2)$] Every point $x$ of $\mathcal{H}$ has Type A, and the valuation $f_x$ is classical with center $x$.  
\item[$(3)$] Every point $y$ at distance $1$ from $\mathcal{H}$ has Type C and is collinear with a unique point $y'$ of $\mathcal{H}$. Moreover, $\mathcal{O}_{f_y} = \{ y' \}$.
\end{enumerate}
\end{lem}
\begin{proof}
If $x$ is a point of $\mathcal{H}$, then $f_x(u) = \dist_{\mN}(x,u) = \dist_{\mH}(x,u)$ for every point $u$ of $\mH$, showing that $f_x$ is classical with center $x$. If $y$ is a point not contained in $\mH$, then the maximal distance from $y$ to a point of $\mH$ is equal to $\dist(y,\mH) + M_{f_y}$. Since this maximal distance is at most 3, we should have $M_{f_y} \leq 2$ and so $y$ should have Type C or D by Table \ref{tab:1}. If $y$ has Type C or D, then $M_{f_y}=2$ and so the fact that the maximal distance is at most 3 implies that $\dist(y,\mathcal{H})=1$. 

We show that $y$ cannot be of Type D. Suppose on the contrary that $y$ is of Type D. From Table~\ref{tab:1} we see that there are five points with $f_y$-value $0$ giving rise to five points inside $\mH$ collinear with $y$, i.e. five lines through $y$ intersecting $\mathcal{H}$ in a point. By Table \ref{tab:2}, each of these lines has Type ADD. The five lines of $\mathcal{V}$ induced by these five lines of $\mH$ are mutually distinct since the five classical valuations contained in them are mutually distinct (as their centers are distinct). A contradiction follows from the fact that through a given $\mV$-point of Type D, there are only three distinct $\mV$-lines of Type ADD, see Table \ref{tab:2}. 
\end{proof}

\begin{cor} \label{co3.4}
$\mN$ has only points of Types A and C, and only lines of Types AAA, ACC and CCC. Moreover, every point of Type C is incident with a unique line of Type ACC.
\end{cor}

\begin{lem} \label{lem3.5}
If $L = \{ x_1,x_2,x_3 \}$ is a line of $\mN$, then $f_{x_1}$, $f_{x_2}$ and $f_{x_3}$ are mutually distinct and hence $\{ f_{x_1},f_{x_2},f_{x_3} \}$ is a $\mathcal{V}$-line.
\end{lem}
\begin{proof}
Suppose to the contrary that $f_{x_1}=f_{x_2}=f_{x_3}$. If these valuations are of Type A, then their centers $x_1,x_2,x_3$ would coincide, an obvious contradiction. If these valuations have Type C, then the unique point $y$ for which $f_{x_1}(y) = f_{x_2}(y) = f_{x_3}(y)=0$ would lie at distance 1 from each of $x_1$, $x_2$ and $x_3$, violating Property (NP2).
\end{proof}

\begin{lem} \label{lem3.6}
Suppose there are no quads meeting $\mH$. Then:
\begin{itemize}
\item[$(1)$] Every point of $\Gamma_1(\mathcal{H})$ is incident with precisely nine lines.
\item[$(2)$] If $x$ is a point of $\Gamma_1(\mathcal{H})$ and $L_1,L_2,\ldots,L_8$ denote the eight lines of Type CCC through $x$, then the eight $\mathcal{V}$-lines $\psi(L_1),\psi(L_2),\ldots,\psi(L_8)$ are precisely the eight $\mathcal{V}$-lines of Type CCC through $\psi(x)=f_x$ (see Table \ref{tab:2}).
\item[$(3)$] Every valuation of Type C is induced equally many times by a point of $\Gamma_1(\mathcal{H})$.  
\end{itemize}
\end{lem}
\begin{proof}
Let $x$ be a point not contained in $\mH$ and let $x'$ be the unique point of $\mH$ collinear with $x$. Let $Y$ denote the set of neighbors of $x$ not on the line $xx'$. Every point $y \in Y$ is collinear with a unique point $y' \in \mathcal{H}$. The point $y'$ lies at distance 2 from $x$ and cannot be collinear with $x'$, otherwise $x'$ and $y$ would be two distinct common neighbors of $x$ and $y'$, implying that $x$ and $y'$ would be contained in some quad meeting $\mathcal{H}$. So, $y' \in Y'$, where $Y'$ denotes the set of all points of $\Gamma_2(x) \cap \mathcal{H}$ noncollinear with $x'$. The points in $\mathcal{H}$ at distance 2 from $x$ are those with $f_x$-value 1 and there are 22 such points by Table \ref{tab:1}. Since six of these are collinear with $x'$, we have $|Y'|=16$. The map $\phi: Y \to Y'; y \mapsto y'$ must be a bijection.
Indeed, if $y' \in Y'$, then $\dist(x,y')=2$ and since there are no quads containing $x$ and $y'$, the points $x$ and $y'$ have a unique common neighbor $y$. This point $y$ cannot be contained in $\mathcal{H}$, otherwise it would coincide with $x'$ (as $y \sim x$), in contradiction with the fact that no point of $Y'$ is collinear with $x'$. Since $y$ is not contained in $\mathcal{H}$, we have $\phi(y)=y'$. Moreover, any point $z \in Y$ satisfying $\phi(z)=y'$ must be a common neighbor of $x$ and $y'$ and hence coincide with $y$. So, $\phi$ is a bijection, implying that the total number $\frac{|Y|}{2}$ of lines of Type CCC through $x$ is equal to $\frac{|Y'|}{2}=8$. This proves the first claim.

\medskip The 16 neighbors $y$ of $x$ not contained in $xx'$ all induce distinct valuations of Type C since they give rise to 16 distinct singletons $\mathcal{O}_{f_y}$ corresponding to the 16 points of $Y'$. So, the 8 lines of Type CCC of $\mathcal{V}$ through $f_x$ corresponding to the 8 lines of Type CCC of $\mN$ through $x$ are all distinct. By Table \ref{tab:2}, these are all the lines of $\mathcal{V}$ through the point $f_x$. This proves the second claim.

\medskip In view of Lemma \ref{lem3.1}(a), it suffices to prove that if $f_1$ and $f_2$ are two valuations of Type C contained in some line $\{ f_1,f_2,f_3 \}$ of $\mathcal{V}'$, then $N_1=N_2$, where $N_i$ with $i \in \{ 1,2 \}$ is the number of points of Type C of $\mathcal{N}$ inducing $f_i$. That this is indeed the case follows from Claim (2), which implies that the total number of lines of Type CCC of $\mathcal{N}$ inducing $\{ f_1,f_2,f_3 \}$ is equal to both $N_1$ and $N_2$.
\end{proof}

\bigskip \noindent The following is a consequence of Lemma \ref{lem3.6}.

\begin{cor}
There are no semi-finite generalized hexagons of order $(2,t)$ containing $H^D(2)$ as a full subgeometry.
\end{cor}

\bigskip \noindent \textbf{Remark.} Suppose $\mathcal{H}$ is embedded as a subhexagon in the dual twisted triality hexagon $T(2,8)$. Then every point $x$ of $T(2,8)$ induces a valuation $f_x'$ of $\mathcal{H}$. As $T(2,8)$ contains 819 points, each valuation of Type C of $\mathcal{H}$ must be induced by precisely three points of $T(2,8)$. Then also every $\mathcal{V}$-line of Type CCC must be induced by precisely three lines of $T(2,8)$. These facts in combination with the fact that $T(2,8)$ has no subgeometries that are ordinary $k$-gons with $k \in \{ 3,4,5 \}$ can be used to prove  several properties of the valuation geometry $\mathcal{V}$. For instance, the property mentioned in Lemma \ref{lem3.1}(b) can be proved in this way. Indeed, take a line $\{ x_1,x_2,x_3  \}$ of Type CCC in $T(2,8)$ such that $f_{x_1}'=f_1$ and $f_{x_2}'=f_2$, and let $y_1$ and $y_2$ be the unique points contained in $\mathcal{O}_{f_1}$ and $\mathcal{O}_{f_2}$. Then the fact that every cycle containing the consecutive vertices $y_1,x_1,x_2,y_2$ has length at least 6 implies that $\dist(y_1,y_2)=3$. The nonexistence of subgeometries that are ordinary $k$-gons with $k \in \{ 3,4,5 \}$ can also be used to prove that the geometry $\mathcal{V}'$ does not have triangles as subgeometries. Another application of the technique will be given in the proof of the following lemma.

\begin{lem} \label{lem3.8}
There are no quads in $\mN$ containing a point of $\mH$. 
\end{lem}
\begin{proof}
Suppose that there exists a quad meeting $\mN$, then there also exists a $(3 \times 3)$-subgrid $G$ meeting $\mH$. Since $G$ and $\mH$ are subspaces of $\mN$, the intersection $G \cap \mathcal{H}$ must be a subspace contained in $G$. We can therefore distinguish the following possibilities:

\medskip (1) The grid $G$ is contained in the generalized hexagon $\mathcal{H}$. This is obviously impossible.

\medskip (2) The grid $G$ intersects $\mathcal{H}$ in the union $L_1 \cup L_2$ of two intersecting lines $L_1$ and $L_2$. But this is also impossible. If this were the case, then any point of $G \setminus (L_1 \cup L_2)$ would be collinear with at least two points of $\mathcal{H}$, namely one of $L_1$ and another one on $L_2$. This would be in contradiction with Lemma \ref{lem3.3}(3).

\medskip (3) The grid $G$ intersects $\mathcal{H}$ in a set of 2 or 3 mutually noncollinear points. Then take a point in $G$ that is collinear with two points of $G \cap \mathcal{H}$. As in (2) this point would be collinear with at least two points of $\mathcal{H}$, again a contradiction.  

\medskip (4) The grid $G$ intersects $\mathcal{H}$ in a line $L$. Let $L'$ be a line of $G$ disjoint from $L$, let $x_1$ and $x_2$ be two distinct points of $L$, and let $y_i$ with $i \in \{ 1,2 \}$ be the unique point of $L$ collinear with $x_i$. The facts that $L$ is a line of Type CCC and $y_1 \sim y_2$ would be in contradiction with Lemma \ref{lem3.1}(b).

\medskip (5)  The grid $G$ intersects $\mathcal{H}$ in a unique point $x$. We label the points of $G$ by $x_{ij}$, $i,j \in \{ 1,2,3 \}$, such that $x=x_{33}$ and $x_{ij} \sim x_{i'j'}$ if and only if either $i=i'$ or $j=j'$ (see Figure \ref{fig}). Regard $\mH$ as a subhexagon of the dual twisted triality hexagon $T(2,8)$.

Suppose $f_{x_{31}} \not= f_{x_{13}}$. Let $u$ be a point of $T(2,8)$ for which $f_u' = f_{x_{11}}$. Then there exist unique points $v_1,v_2$ in $T(2,8)$ collinear with $u$ such that $uv_1 \not= uv_2$, $f_{v_1}' = f_{x_{31}}$ and $f_{v_2}' = f_{x_{13}}$. Then $u,v_1,x,v_2,u$ define a subquadrangle of $T(2,8)$, which is impossible. Hence, $f_{x_{31}} = f_{x_{13}}$.

If we repeat the above argument with $x_{11}$ replaced by $x_{12}$, then we find that $f_{x_{13}} = f_{x_{32}}$. So, $f_{x_{31}} = f_{x_{32}}$. But this is in contradiction with the fact that $\{ f_{x_{31}},f_{x_{32}},f_{x_{33}} \}$ is a line of $\mathcal{V}$, see Lemma \ref{lem3.5}.
\end{proof}

\begin{figure}
\begin{center}
\begin{tikzpicture}[line cap=round,line join=round,>=triangle 45,x=1.0cm,y=1.0cm]
\draw (2.8226839826839827,2.7532467532467533)-- (1.4026839826839828,1.393246753246754);
\draw (1.4026839826839828,1.393246753246754)-- (2.762683982683982,-0.026753246753246085);
\draw (2.762683982683982,-0.026753246753246085)-- (4.182683982683982,1.3332467532467531);
\draw (4.182683982683982,1.3332467532467531)-- (2.8226839826839827,2.7532467532467533);
\draw (2.1126839826839827,2.0732467532467536)-- (3.4683497871330253,0.6490956927190765);
\draw (3.5026839826839824,2.0432467532467533)-- (2.086723880677817,0.6790286244590737);
\draw [rotate around={-0.26525634337598575:(2.8453178777145745,-0.5793713674254614)}] (2.8453178777145745,-0.5793713674254614) ellipse (2.169575869868613cm and 1.0998176043554964cm);
\begin{scriptsize}
\draw [fill=black] (2.8226839826839827,2.7532467532467533) circle (1.5pt);
\draw[color=black] (2.959757021310607,2.966729374285681) node {$x_{11}$};
\draw [fill=black] (1.4026839826839828,1.393246753246754) circle (1.5pt);
\draw[color=black] (1.5410071396897176,1.612468123647558) node {$x_{31}$};
\draw[color=black] (1.6,2.3) node {$G$};
\draw [fill=black] (2.762683982683982,-0.026753246753246085) circle (1.5pt);
\draw[color=black](2.8,-0.03) node[anchor=north west] {$x = x_{33}$};
\draw [fill=black] (4.182683982683982,1.3332467532467531) circle (1.5pt);
\draw[color=black] (4.3269159981452825,1.5479794926647903) node {$x_{13}$};
\draw [fill=black] (2.1126839826839827,2.0732467532467536) circle (1.5pt);
\draw[color=black] (2.250382080500162,2.2960476120648963) node {$x_{21}$};
\draw [fill=black] (3.4683497871330253,0.6490956927190765) circle (1.5pt);
\draw[color=black] (3.604643331138284,0.8644000042474519) node {$x_{23}$};
\draw [fill=black] (3.5026839826839824,2.0432467532467533) circle (1.5pt);
\draw[color=black] (3.6433365097279444,2.2573544334752356) node {$x_{12}$};
\draw [fill=black] (2.086723880677817,0.6790286244590737) circle (1.5pt);
\draw[color=black] (2.224586628107055,0.9030931828371125) node {$x_{32}$};
\draw [fill=black] (2.7947039316808997,1.3611376888529136) circle (1.5pt);
\draw[color=black] (2.9339615689175,1.5866726712544508) node {$x_{22}$};
\draw[color=black] (1.8,0.10) node {$\mathcal{H}$};
\end{scriptsize}
\end{tikzpicture}
\caption{Lemma~\ref{lem3.8} (5)} \label{fig}
\end{center}
\end{figure}

\begin{cor} \label{co3.9}
Every point of $\mN$ not contained in $\mH$ is incident with precisely nine lines.
\end{cor}
\begin{proof}
This follows from Lemmas \ref{lem3.6} and \ref{lem3.8}.
\end{proof}

\begin{lem} \label{lem3.10}
Let $x$ be a point in $\mN$ not contained in $\mH$, $x'$ the unique point of $\mH$ collinear with $x$ and $y$ a point of $\mathcal{H}$ at distance $2$ from $x'$. Then $\dist(x,y)=3$ and every neighbor of $y$ has at most one common neighbor with $x$.  
\end{lem}
\begin{proof}
Let $u$ denote the unique common neighbor of $x'$ and $y$. Then $\dist(x,u)=2$. If $\dist(x,y) \not= 3$, then $\dist(x,y) \leq 2$ and (NP2) would imply that the line $uy$ contains a point at distance 1 from $x$, in contradiction with the fact that $\mathcal{O}_{f_x}$ is a singleton. So, $\dist(x,y)=3$.

Suppose $z$ is a neighbor of $y$ such that $x$ and $z$ have more than one common neighbor. Then $x$ and $z$ are contained in a unique quad and hence also in some $(3 \times 3)$-subgrid $G$. By Lemma \ref{lem3.8}, $Q \cap \mathcal{H} = \emptyset$ and hence also $G \cap \mathcal{H} = \emptyset$. If two points of $G$, say $x_1$ and $x_2$, would induce the same valuation $f$, and $\mathcal{O}_f = \{ y \}$, then $y \in \mathcal{H}$ would, as a common neighbor of $x_1$ and $x_2$, be contained in $Q$, in contradiction with $Q \cap \mathcal{H} = \emptyset$. So, the nine points and six lines of $G$ induce nine distinct points and six distinct lines of $\mathcal{V}'$, i.e. a $(3 \times 3)$-subgrid $G'$ of $\mathcal{V}'$. By Lemma \ref{lem3.1}(b) applied to the opposite points $f_{x}$ and $f_z$ of $G'$, we now see that $x'$ and $y$ should be at distance 3 from each other, a contradiction. 
\end{proof}

\begin{lem} \label{lem3.11}
Every point $x$ of $\mathcal{H}$ is incident with precisely nine lines. As a consequence, $\mathcal{N}$ has order $(2,8)$. 
\end{lem}
\begin{proof}
Let $y$ be a point of $\mathcal{H}$ at distance 2 from $x$, and let $f$ be a valuation of Type C for which $\mathcal{O}_f=\{ y \}$. By Lemma \ref{lem3.6}(3), there exists a point $z$ in $\mathcal{N}$ not contained in $\mathcal{H}$ such that $f_z=f$. Then $z$ is collinear with $y$. By Lemma \ref{lem3.10}, $\dist(x,z)=3$. Since there are no quads through the point $x \in \mathcal{H}$, each point of $\Gamma_1(z) \cap \Gamma_2(x)$ has a unique neighbor with $x$. We count the total number of paths $z,u,v,x$ of length 3 connecting the points $z$ and $x$. Since there are nine lines through $z$ and each of these lines contains a unique point at distance 2 from $x$, there are nine possibilities for $u$. For each such $u \in \Gamma_2(x)$, there is a unique choice for $v$, namely the unique neighbor of $u$ and $x$. So, the number of desired paths equals 9. On the other hand, we also see that the number of possible choices for $v$ is equal to $t_x+1$, the total number of lines through $x$ (which could be infinite). Indeed, each line through $x$ contains a unique point $v$ at distance 2 from $z$. By Lemma \ref{lem3.10}, each such point $v$ has a unique common neighbor $u$ with $z$, showing that the number of desired paths is also equal to $t_x+1$. So, the point $x$ is incident with precisely $t_x+1=9$ lines.
\end{proof}

\begin{lem} \label{new}
Let $\mathcal{S}$ be a finite near hexagon of order $(s,t)$ having $v$ points. Then $v \leq (s+1)(s^2t^2+st+1)$, with equality if and only if $\mathcal{S}$ is a generalized hexagon.
\end{lem}
\begin{proof}
Let $x$ be an arbitrary point of $\mathcal{S}$. Then $|\Gamma_0(x)|=1$ and $|\Gamma_1(x)|=s(t+1)$. Since every point of $\Gamma_1(x)$ is collinear with $st$ points of $\Gamma_2(x)$ and every point of $\Gamma_2(x)$ is collinear with at least one point of $\Gamma_1(x)$, we have $|\Gamma_2(x)| \leq |\Gamma_1(x)| \cdot st = s^2t(t+1)$. Every point of $\Gamma_2(x)$ is collinear with at most $st$ points of $\Gamma_3(x)$, and every point $y \in \Gamma_3(x)$ is collinear with precisely $t+1$ points of $\Gamma_2(x)$ (one on each line through $y$), showing that $|\Gamma_3(x)| \leq \frac{|\Gamma_2(x)| \cdot st}{t+1} \leq \frac{s^2t(t+1)}{t+1} st = s^3t^2$. It follows that $v \leq 1+ s(t+1) + s^2t(t+1) + s^3 t^2 = (s+1)(s^2t^2+st+1)$. Equality holds if and only if $|\Gamma_1(x) \cap \Gamma_1(y)| =1$ for every point $y \in \Gamma_2(x)$. As the above holds for any point $x$ of $\mathcal{S}$, we see that the lemma must hold. 
\end{proof}

\begin{thm}
The near hexagon $\mathcal{N}$ is a generalized hexagon and hence isomorphic to $T(2,8)$.
\end{thm}
\begin{proof}
We know that $\mathcal{N}$ has order $(2,8)$. So, in order to show that $\mathcal{N}$ is a generalized hexagon, it suffices by Lemma \ref{new} to prove that $\mathcal{N}$ has 819 points. This indeed holds. Every point outside $\mathcal{H}$ is collinear with a unique point of $\mathcal{H}$, and every point of $\mathcal{H}$ is collinear with precisely $2 \cdot (9-3)$ points of $\Gamma_1(\mathcal{H})$. So, the total number of points of $\mN$ is equal to $|\mH| + |\mH| \cdot 12 = 819$.
\end{proof}

\section{Near hexagons containing an isometrically embedded $H(2)$} \label{sec4}

In this section, we study near hexagons of order $(2,t)$ that contain the generalized hexagon $H(2)$ as an isometrically embedded subgeometry. 

\subsection{The valuation geometry of $H(2)$} \label{sec4.1}

The valuation geometry $\mV$ of $H(2)$ is defined as the partial linear space whose points are the valuations of $H(2)$ and whose lines are all the triples $\{ f_1,f_2,f_3 \}$, where $f_1$, $f_2$ and $f_3$ are three mutually distinct valuations of $H(2)$ such that $f_1$ and $f_2$ are two neighboring valuations and $f_3 = f_1 \ast f_2$. Again, we have used the computer algebra system GAP to determine all valuations and all lines of the valuation geometry $\mathcal{V}$, see \cite{Bi-bdb} and the appendix. The results are summarized in Tables \ref{tab:3} and \ref{tab:4}, where we have followed the same notational conventions as before. There are seven isomorphism classes of valuations and twenty line types in $\mV$. Unlike $H^D(2)$, we do have ovoids in $H(2)$ which give rise to an ovoidal (Type $C$) valuation. 

\begin{table}[!htbp]
\begin{center}
\begin{tabular}{ | l || l | l | l | l | l | }
 \hline
  Type & $\#$ & $M_f$ & $|\mO_f|$ & $|H_f|$  & Value Distribution \\ \hline \hline
$A$ & $63$ & $3$ & $1$ & $31$ & $[ 1, 6, 24, 32 ]$ \\ \hline
$B_1$ & $126$ & $2$ & $1$ & $23$ & $[ 1, 22, 40, 0 ]$ \\ \hline
$B_2$ & $252$ & $2$ & $3$ & $27$ & $[ 3, 24, 36, 0 ]$ \\ \hline
$B_3$ & $504$ & $2$ & $4$ & $29$ & $[ 4, 25, 34, 0 ]$ \\ \hline
$B_4$ & $72$ & $2$ & $7$ & $35$ & $[ 7, 28, 28, 0 ]$ \\ \hline
$B_5$ & $378$ & $2$ & $9$ & $39$ & $[ 9, 30, 24, 0 ]$ \\ \hline
$C$ & $36$ & $1$ & $21$ & $21$ & $[ 21, 42, 0, 0 ]$ \\ \hline
\end{tabular}
\end{center}
\caption{The valuations of $H(2)$}
\label{tab:3}
\end{table}

\begin{table}[!htbp]
\begin{center}
\begin{tabular}{| l || l | l | l | l | l | l | l |}
\hline 
Type & $A$ & $B_1$ & $B_2$ & $B_3$ & $B_4$ & $B_5$ & $C$ \\ \hline \hline
$AAA$& $3$& - & - & - & - & - & - \\ \hline 
$AB_1B_1$& $1$& $1$& - & - & - & - & - \\ \hline 
$AB_2B_2$& $6$& - & $3$& - & - & - & - \\ \hline 
$AB_3B_3$& $16$& - & - & $4$& - & - & - \\ \hline 
$AB_4B_4$& $4$& - & - & - & $7$& - & - \\ \hline 
$AB_5B_5$& $3$& - & - & - & - & $1$& - \\ \hline 
$B_1B_1B_1$& - & $3$& - & - & - & - & - \\ \hline 
$B_1B_1B_2$& - & $16$& $4$& - & - & - & - \\ \hline 
$B_1B_1B_5$& - & $6$& - & - & - & $1$& - \\ \hline 
$B_1B_2B_4$& - & $4$& $2$& - & $7$& - & - \\ \hline 
$B_1B_3B_3$& - & $12$& - & $6$& - & - & - \\ \hline 
$B_1B_3C$& - & $12$& - & $3$& - & - & $42$\\ \hline 
$B_2B_2B_2$& - & - & $12$& - & - & - & - \\ \hline 
$B_2B_2B_5$& - & - & $6$& - & - & $2$& - \\ \hline 
$B_2B_3B_3$& - & - & $10$& $10$& - & - & - \\ \hline 
$B_2CC$& - & - & $1$& - & - & - & $14$\\ \hline 
$B_3B_3B_5$& - & - & - & $3$& - & $2$& - \\ \hline 
$B_4B_4C$& - & - & - & - & $1$& - & $1$\\ \hline 
$B_5B_5B_5$& - & - & - & - & - & $1$& - \\ \hline 
$B_5CC$& - & - & - & - & - & $1$& $21$\\ \hline 
\end{tabular}
\end{center}
\caption{The lines of the valuation geometry of $H(2)$} 
\label{tab:4}
\end{table}

\subsection{Proof of those main results that involve a subhexagon $H(2)$} \label{sec4.2}

In this section $\mN$ denotes a near hexagon of order $(2,t)$ that has a subhexagon $\mH$ isomorphic to $H(2)$ isometrically embedded in it as a full subgeometry. We will use the same convention as before regarding the types for the points and lines of $\mN$. We then have:

\begin{lem} \label{lem4.1}
Every point of $\mN$ has one of the seven types mentioned in Table~\ref{tab:3}. For a line $L$ of $\mN$ if all the points on $L$ induce distinct valuations then $L$ has one of the twenty types mentioned in Table~\ref{tab:4}. 
\end{lem}

\begin{lem} \label{lem4.2}
Every point of $\mN$ is at distance at most $2$ from $\mH$. Moreover, 
\begin{itemize}
\item points in $\mH$ are of Type $A$ (classical);

\item points at distance $1$ from $\mH$ are of Type $B_i$ for some $i \in \{1,2,3,4,5\}$;

\item points at distance $2$ from $\mH$ are of Type $C$ (ovoidal). 
\end{itemize}
\end{lem}
\begin{proof}
If $x$ is a point of $\mathcal{N}$, then the maximal distance from $x$ to a point of $\mH$ is equal to $\dist(x,\mH) + M_{f_x}$. Since $\dist(x,\mH) + M_{f_x} \leq 3$ and $M_{f_x} \geq 1$, we have $\dist(x,\mathcal{H}) \leq 2$. If $\dist(x,\mathcal{H})=2$, then $M_{f_x}=1$ and so $f_x$ and $x$ have Type C. If $x \in \mathcal{H}$, then $f_x$ is a classical valuation since $\mathcal{H}$ is isometrically embedded into $\mathcal{N}$. Finally, suppose that $\dist(x,\mathcal{H})=1$. Then $M_{f_x} \leq 2$ and so $x$ has Type C or $B_i$ for some $i \in \{ 1,2,3,4,5 \}$. We prove that the former case cannot occur. Indeed, if $x$ were ovoidal, then $\mathcal{O}_{f_x}$ would be an ovoid, and there would exist two points $y_1$ and $y_2$ in $\mathcal{O}_{f_x}$ at distance 3 from each other. (Indeed, starting from a point $y_1 \in \mathcal{O}_{f_x}$, select points $u$, $v$ and $y_2$ of $\mH$ such that $u \in \Gamma_1(y_1)$, $v \in \Gamma_1(u) \setminus \mathcal{O}_{f_x}$ with $y_1u \not= uv$ and $y_2 \in \Gamma_1(v) \cap \mathcal{O}_{f_x}$ with $vy_2 \not= uv$). Since $\dist_{\mH}(y_1,y_2)=3$, we also have $\dist_{\mN}(y_1,y_2)=3$, but this would be in contradiction with the fact that $x$ is a common neighbor of $y_1$ and $y_2$.
\end{proof}

\begin{lem} \label{lem4.3}
If $L = \{ x_1,x_2,x_3 \}$ is a line of $\mN$, then $f_{x_1}$, $f_{x_2}$ and $f_{x_3}$ are mutually distinct and hence $\{ f_{x_1},f_{x_2},f_{x_3} \}$ is a line of $\mV$.
\end{lem}
\begin{proof}
Suppose $f:=f_{x_1}=f_{x_2}=f_{x_3}$ and let $y$ be an arbitrary point of $\mathcal{O}_f$. By Lemma \ref{lem4.2}, $\dist(y,x_1)=\dist(y,x_2)=\dist(y,x_3)$, but this would be in contradiction with (NP2). 
\end{proof}

\begin{thm} \label{theo4.4}
There exists no generalized hexagon of order $(2,t)$ containing $H(2)$ as an isometrically embedded proper subgeometry. 
\end{thm}
\begin{proof}
We suppose that $\mN$ is a generalized hexagon. We show that all points of $\mN$ are contained in $\mH$ by the following sequence of steps.
\begin{enumerate}
\item \textit{No point of $\mN$ has Type $B_i$ with $i > 1$.}
      Let $x$ be such a point. 
      It must necessarily be at distance $1$ from $\mH$ by Lemma~\ref{lem4.2}. 
      By Table~\ref{tab:3} there exist two points $y$ and $z$ in $\mO_{f_x}$ which must necessarily be collinear with $x$. 
      Therefore $d_{\mN}(y,z) = 2$ by (NP2). Hence, $d_{\mH}(y,z) = 2$ and there exists a common neighbor of $y$ and $z$ inside $\mH$ and thus distinct from $x$.
      This contradicts (GH2).
      
\item \textit{There is no point in $\mN$ of Type $C$.} 
      From Table~\ref{tab:4} we see that every line through such a point would contain a point of Type $B_i$ for some $i > 1$, but no such points exist by the previous step.

\item \textit{There is no point in $\mN$ of Type $B_1$.}
      Let $x$ be a point of Type $B_1$ in $\mN$, necessarily at distance $1$ from $\mH$. 
      Since $|\mO_{f_x}| = 1$, there is a unique point $x'$ in $\mH$ collinear with $x$. 
      The valuation $f_x$ of $\mH$ has $22$ points of value $1$. 
      This is equivalent to the statement that there are $22$ points in $\mH$ at distance $2$ from $x$ since $f_x(y) = d(x,y) - 1$ for all points $y$ in $\mH$.  
      Six of these $22$ points come from the neighbors of $x'$ in $\mH$ and hence the remaining sixteen points must have a common neighbor with $x$ which neither lie in $\mH$ nor on the line $xx'$. 
      These common neighbors give rise to sixteen points of Type $B_1$ collinear with $x$. 
      These sixteen points induce sixteen distinct Type $B_1$ valuations each of which is incident together with $f_x$ with a $\mathcal{V}$-line of Type $B_1B_1B_1$.
      But, from the entry in row $B_1B_1B_1$ and column $B_1$ in Table~\ref{tab:4} we see that there can be at most six such Type $B_1$ valuations. \\
\end{enumerate}      
Therefore all points of $\mN$ have Type $A$ and hence by Lemma~\ref{lem4.2} they are all contained in $\mH$. This shows that $\mN$ is in fact equal to $\mH$.
\end{proof}

\bigskip \noindent Theorem~\ref{theo4.4} completes the proof of the \textit{non-existence of semi-finite generalized hexagons of order $(2,t)$ containing a subhexagon of order $2$.}  We can in fact prove more. We show below that every near hexagon $\mN$ containing a generalized hexagon $\mH \cong H(2)$ isometrically and fully embedded in it must also be finite. In fact, the reasoning below would give an alternative proof of Theorem \ref{theo4.4}, if one would invoke the classification of all finite generalized hexagons of order $(2,t)$ (\cite{CT}) and the observation that none of these finite generalized hexagons contains $H(2)$ as a proper subhexagon.

\begin{lem} \label{lem4.5}
For every $i \in \{ 2,3,4,5 \}$, there are only finitely many points in $\mN$ of Type $B_i$. 
\end{lem}
\begin{proof}
Let $x$ be a point of Type $B_i$ for some $i > 1$.  From Table~\ref{tab:3} there are at least two distinct points $y$ and $z$ of $\mH$ in $\mO_{f_x}$, necessarily collinear with $x$ by Lemma~\ref{lem4.2}. By (NP2) $y$ and $z$ must be at distance $2$ from each other. Now, $y$ and $z$ have at most five common neighbors (Lemma \ref{lem2.1}) and one of these must be contained in $\mH$. From this it follows that the number of points of Type $B_i$ for some $i \in \{ 2,3,4,5 \}$ is at most $4$ times the number of unordered pairs $\{ p,q \}$ of points of $\mH$ at distance 2 from each other, i.e. at most 3024.
\end{proof}

\begin{lem} \label{lem4.6}
There are only finitely many points of Type $B_1$ in $\mN$.
\end{lem}
\begin{proof}
Let $\mathcal{B}$ denote the set of those points of $\mathcal{N}$ that have type $B_i$ for some $i \in \{ 2,3,4,5 \}$. Then $\mathcal{B}$ is finite by Lemma~\ref{lem4.5}. Let $\mathcal{A}$ denote the set of those points of $\mathcal{N}$ that have type $A$, i.e., the points of $\mathcal{H}$. Then the set $\mathcal{A} \cup \mathcal{B}$ is also finite. Let $x$ be a point of type $B_1$ in $\mathcal{N}$. Then by Lemma \ref{lem4.2}, $x$ is at distance 1 from $\mathcal{H}$, and since $\mathcal{O}_{f_x}$ is a singleton, there exists a unique point $\pi(x)$ in $\mathcal{H}$ collinear with $x$. If $x$ is only collinear with points of type $A$, $B_1$ or $C$, then by the same reasoning as in the proof of Theorem \ref{theo4.4}, we get a contradiction. So, $x$ is collinear with at least one point of $\mathcal{B}$, and we have already seen that it is collinear with at least one point of $\mathcal{A}$. Thus $x$ is the common neighbour of two points at distance 2 in the finite set $\mathcal{A} \cup \mathcal{B}$. Since each such pair of points at distance 2 in the near polygon $\mathcal{N}$ has finitely many (at most five) common neighbours, we see that the set of points of type $B_1$ must be finite; in fact, the cardinality of this set is bounded by five times the number of unordered pairs of points at distance 2 in $\mathcal{A} \cup \mathcal{B}$.
\end{proof}

\begin{lem} \label{lem4.7}
There are only finitely many points of Type C in $\mN$.
\end{lem}
\begin{proof}
Let $x$ be a point of type $C$ in $\mathcal{N}$. Then the set of points of $\mathcal{H}$ at distance 2 from $x$ is a 1-ovoid of $\mathcal{H}$ and hence it has cardinality 21. Let $S_x$ be the set of common neighbours between $x$ and the elements of $\mathcal{O}_{f_x}$ (the 1-ovoid of $\mathcal{H}$ induced by $x$). By Lemma \ref{lem4.2}, each element $y$ of $S_x$ has type $B_i$ for some $i \in \{ 1,2,\ldots,5 \}$ and hence by Table \ref{tab:3}, $y$ is collinear with at most nine points of $\mathcal{H}$. Therefore, $|S_x| \geq \frac{21}{9}$, and we get two points of the set $\Gamma_1(\mathcal{H})$ at distance 2 from each other having $x$ as a common neighbour. By Lemma \ref{lem4.5} and Lemma \ref{lem4.6}, the set $\Gamma_1(\mathcal{H})$ is finite. A similar reasoning as in the proof of Lemma \ref{lem4.6} then shows that there are only finitely many points of type $C$ in $\mathcal{N}$.
\end{proof}

\bigskip \noindent Observe also that the Type A points are precisely the points of $\mathcal{H}$. So, from Lemmas \ref{lem4.5}, \ref{lem4.6} and \ref{lem4.7}, we can conclude the following:

\begin{cor} \label{co4.8}
There does not exist any infinite near hexagon of order $(2,t)$ that has an isometrically embedded subgeometry isomorphic to $H(2)$. 
\end{cor}

\appendix

\section{Computer computations}

In this appendix, we briefly discuss the algorithms and GAP code we used to obtain the computer results described in Sections \ref{sec3.1} and \ref{sec4.1}. The whole computer code can be found in \cite{Bi-bdb} (there are two independent GAP codes, written by the two authors).

In order to do computer calculations inside the generalized hexagons, we need computer models for them. There are two generalized hexagons of order $2$, the split Cayley hexagon $H(2)$ and its dual $H^D(2)$. Both of them have $63$ points. The full automorphism group for each of them  acts primitively and distance-transitively on the points. Therefore, these permutation groups can be found in the GAP (cf.~\cite{gap}) library of primitive groups. 

For $H(2)$, the following code implements the action of its automorphism group \verb|g| on the point set (which is identified here with the set $\{ 1,2,\ldots,63 \}$):
\begin{verbatim}
g:=AllPrimitiveGroups(DegreeOperation,63)[4];
\end{verbatim}
For $H^D(2)$, the following code implements the action of its automorphism group \verb|g| on the point set (which is also identified here with the set $\{ 1,2,\ldots,63 \}$):
\begin{verbatim}
g:=AllPrimitiveGroups(DegreeOperation,63)[2];
\end{verbatim}
Having computer models for these permutation groups, it is not hard any more to implement computer models for the associated generalized hexagons (with point set $\{ 1,2,\ldots,$ $63 \}$, line set \verb|lines| and distance function \verb|dist(.,.)|):

\begin{verbatim}
orbs := Orbits(Stabilizer(g,1),[1..63]);
dist1 := Filtered(orbs,x->Size(x)=6)[1];
dist2 := Filtered(orbs,x->Size(x)=24)[1];
dist3 := Filtered(orbs,x->Size(x)=32)[1];
partition := [[1],dist1,dist2,dist3];
perp := Union([1],dist1);
r := RepresentativeAction(g,1,partition[2][1]);
line := Intersection(perp,OnSets(perp,r));
lines := Orbit(g,line,OnSets);
DistMat := NullMat(63,63);
for x in [1..63] do
 r := RepresentativeAction(g,x,1);
 for y in [1..63] do
  z := y^r; 
  i := 1; while not(z in partition[i]) do i := i+1; od;
  DistMat[x][y] := i-1;
 od;
od;
dist := function(x,y) 
return DistMat[x][y]; 
end;
\end{verbatim}

\medskip \noindent Now that we have computer models for the generalized hexagons, our next goal will be to determine the valuations in each of them. Since the number of points is relatively small (63),  these valuations can be determined with the aid of a backtrack algorithm. However, to create valuations of general point-line geometries with three points per line we would like to follow a uniform approach, an approach which not only seems more efficient to us, but would also work (i.e., produce results in reasonable time) for larger nice geometries.

The idea is as follows. With every valuation $f$ there is an associated hyperplane $H_f$. For partial linear spaces $\mathcal{S}$ having three points per line, it is an easy task (computationally) to create hyperplanes in a fast and direct way, without going through a potentially long backtrack process. The method to create these hyperplanes directly is based on the observation that a proper set $H$ of points of such a partial linear space $\mathcal{S}$ is a hyperplane if and only if the characteristic vector of its complement $\overline{H}$ is orthogonal over $\mathrm{GF}(2)$ with the characteristic vectors of all lines. The second author learned about this direct and efficient method for creating hyperplanes from Sergey Shpectorov a few years ago. This direct method can still easily produce all hyperplanes of certain geometries having more than 1000 points (a number which seems far too big for a backtrack search).

Based on the above method, Algorithm \ref{algo1} below determines representatives for the various isomorphism classes of hyperplane complements of the generalized hexagon $\mathcal{H} \in \{ H(2),H^D(2) \}$. 

\begin{algorithm}[!htbp]
\caption{Generating the hyperplane complements}
\label{algo1}
\begin{algorithmic}
\STATE \textbf{input:} the incidence matrix $M$
\STATE \textbf{output:} list of distinct representatives of the isomorphism classes of hyperplane complements
\STATE
\STATE $U \gets$ nullspace of $M$ over $\mathbb{F}_2$
\STATE $S \gets$ the set of points at distance $3$ from a fixed point
\COMMENT{such a set is always a hyperplane complement}
\STATE \textit{HypComp} $\gets$ [$S$]
\STATE \textit{balance} $\gets 2^{dim(U)} - 63 - 1$
\WHILE{\textit{balance} $> 0$}
\STATE pick a random non-zero vector in $U$
\STATE find the hyperplane complement $H$ corresponding to it
  \IF{$H$ is not isomorphic to any element of \textit{HypComp}}
  \STATE Add $H$ to \textit{HypComp}
  \STATE \textit{balance} = \textit{balance} $-$ Index$_g$(Stabilizer$_g$($H$))
  \ENDIF
\ENDWHILE
\RETURN \textit{HypComp}
\end{algorithmic}
\end{algorithm}

\medskip \noindent Using Algorithm \ref{algo1}, we found that $H(2)$ has up to isomorphism 25 hyperplanes and that $H^D(2)$ has up to isomorphism 14 hyperplanes. We wish to note here that Frohardt and Johnson \cite{Fr-Jo} also classified all hyperplanes of $H(2)$ and $H^D(2)$, finding the same number of hyperplanes. Now that we have determined all hyperplane complements, we wish to know how many of them are associated with valuations, noting that it is possible that certain hyperplane complements are associated with more than one valuation. We achieve this goal by means of Algorithm \ref{algo2} below.

We define a \textit{partial valuation} of $\mH$ as a function $f$ defined on a subspace $S_f$ of $\mH$ such that $f$ is a semi-valuation on the subgeometry of $\mH$ induced on $S_f$. In Algorithm~\ref{algo2} we use a function \verb|AssignValue| which takes as input a function $f$ defined on a set  $S$ of points, a point $x$ and a value $i$, and finds a partial valuation $g$ such that: (i) $S_g$ is the smallest subspace containing $S$ and $\{ x \}$; (ii) $g(y) = f(y)$ for all $y \in S$; (iii) $g(x) = i$. If such a function $g$ does not exist then \verb|AssignValue| returns \verb|fail|. This function can be implemented easily by using the fact that if two points of a line $L$ have been assigned a value then the value of the third point is uniquely determined. 

\begin{algorithm}[!htbp]
\caption{Create valuations}
\label{algo2}

\begin{algorithmic}
\STATE \textbf{input:} a hyperplane complement $H$
\STATE \textbf{output:} list of all semi-valuations $f$ for which $\comp{H_f} = H$ and for which the maximal value is equal to 0
\STATE
\STATE $val \gets$ function which is $0$ on all points of $H$ and undefined on other points 
\STATE $f = $ \textit{AssignValue($val$, $H[1]$, 0)} 
\STATE \textit{complete} $\gets$ []
\STATE \textit{incomplete} $\gets$ []
\IF {$f$ is defined on all points}
\STATE Add $f$ to \textit{complete} 
\ELSE
\STATE Add $f$ to \textit{incomplete} if $f$ is distinct from \textit{fail}
\ENDIF

\WHILE {\textit{incomplete} is not empty}
\STATE \textit{val} $\gets$ an element popped from \textit{incomplete}
  \STATE $x \gets$ a random point for which \textit{val} is not defined 
  \FOR {$i$ in $\{ -1,-2,-3 \}$} 
  \STATE $f \gets$ \textit{AssignValue(\textit{val}, $x$, i)}
    \IF {$f$ is defined on all points}
    \STATE Add $f$ to \textit{complete} 
    \ELSE
    \STATE Add $f$ to \textit{incomplete} if $f$ is distinct from \textit{fail}
    \ENDIF
  \ENDFOR
\ENDWHILE
\RETURN \textit{complete}
\end{algorithmic}

\end{algorithm}

\medskip \noindent Using Algorithm~\ref{algo2} we found that from the 25 nonisomorphic hyperplanes of $H(2)$, seven are associated with valuations. These valuations can be found in Table \ref{tab:3}, where they have been given a type (Type A, $B_1$, ..., $B_5$, C). Six of these seven hyperplanes are associated with a unique valuation, while the seventh (Type $B_4$) is associated with two valuations. These two valuations are however isomorphic.  

From the 14 nonisomorphic hyperplanes of $H^D(2)$, four are associated with valuations. These valuations can be found in Table \ref{tab:1}, where they have been given a type (Type A, B, C, D). Three of these four hyperplanes are associated with a unique valuation, while the fourth (Type B) is associated with two valuations. Again, these two valuations are isomorphic.

\medskip \noindent We have also written GAP code to verify whether two valuations are isomorphic, to determine the type of a given valuation (based on the properties mentioned in Tables \ref{tab:1} and \ref{tab:3}), to verify whether two valuations are neighboring (with the function \verb|AreNeighboring|) and to determine the valuation $f_1 \ast f_2$ if $f_1$ and $f_2$ are two distinct neighboring valuations (with the function \verb|ThirdFromTwo|). With these functions we were able to verify all the claims made in Sections \ref{sec3.1} and \ref{sec4.1}. We refer to \cite{Bi-bdb} for more details.

To verify certain properties of the valuation geometry $\mathcal{V}$, it might be handy to have a routine to create subgeometries. Given a set $S$ of mutually distinct valuations, the following algorithm computes the valuation geometry $\mV_S$ which is defined as the subgeometry of $\mV$ whose points are the elements of $S$ and whose lines are all the lines of $\mV$ which are completely contained in $S$. The valuation geometry $\mV_S$ is constructed in such a way that its point set is equal to $\{ 1,2,\ldots,|S| \}$.

\begin{algorithm}[!htbp]
\caption{Construct valuation geometry}
\label{algo3}
\begin{algorithmic}
\STATE \textbf{input:} a set $S$ of valuations
\STATE \textbf{output:} points $P$, lines $L$ and the collinearity graph $\Gamma$ of $\mV_S$
\STATE
\STATE $P \gets [1,2,\ldots,Size(S)]$, $L \gets []$, $Adj \gets$ $0$-matrix of size $Size(S) \times Size(S)$ 

\FOR {$\{ i,j \}$ in $\binom{P}{2}$}
 \IF {\textit{AreNeighboring($S[i]$, $S[j]$)}}
 \STATE \textit{val} $\gets$ \textit{ThirdFromTwo($S[i]$, $S[j]$)}
   \IF {\textit{val} $\in S$}
    \STATE $k \gets$ position of \textit{val} in $S$ 
    \STATE Add $\{i,j,k\}$ to $L$
    \STATE $Adj[i][j] = Adj[j][i] = 1$. 
   \ENDIF
 \ENDIF 
\ENDFOR
\STATE construct graph $\Gamma$ from the adjacency matrix $Adj$ 
\RETURN $P$, $L$, $\Gamma$
\end{algorithmic}
\end{algorithm}

\end{document}